\documentclass[12pt, oneside]{amsart}

      \usepackage{amssymb,amsthm}
      \usepackage{float}
      \usepackage{amsmath}
      \usepackage{setspace}
      \usepackage{graphicx}
      \theoremstyle{plain}
      \newtheorem{theorem}{Theorem}[section]
      \newtheorem{lemma}[theorem]{Lemma}
      \newtheorem{corollary}[theorem]{Corollary}
      \newtheorem{proposition}[theorem]{Proposition}

      \theoremstyle{definition}
      \newtheorem{definition}[theorem]{Definition}

      \theoremstyle{remark}
      
      \newtheorem{example}[theorem]{Example}

      \makeatletter
      \def\@setcopyright{}
      \def\serieslogo@{}
      \makeatother

\begin{document}

   \author{Dinesh Udar and Shiksha Saini}
   \address{Department of Applied Mathematics, Delhi Technological University, Delhi, India}
    \email{dineshudar@yahoo.com}
    \email{shiksha96saini@gmail.com}

  \title[ $2- \sqrt{J}U$ rings]{Rings in which the square of a unit is the sum of 1 and an element from $\sqrt{J(R)}$}

\begin{abstract}
Through this paper, we study the rings in which every unit's square is an element of the set $1+\sqrt{J(R)}$, and call them $2-\sqrt{J}U$ rings. Here, $\sqrt{J(R)}=\{x \in R: x^m \in J(R)$ for some $m \geq 1 \}$. We show that every $UU,~UJ,~2-UU,~2-UJ$ and $\sqrt{J}U$ ring is a $2-\sqrt{J}U$ ring. After exploring the basic properties, we show that the corner ring and unit closed subring of a $2-\sqrt{J}U$ ring are also $2-\sqrt{J}U$ rings. The ring of all $n\times n$ matrix rings for any $n>1$ is never a $2-\sqrt{J}U$ ring. We have focused on several other matrix extensions and group rings.
\end{abstract}


% 16N20: Jacobson radical,quasimultiplication, 16N99: Radical (none of the above), 16U99: CONDITION ON ELEMENTS (none of the above), 16U60: units, group of units, 16S50: classification code for Endomorphism rings; matrix rings within the broader category of Associative Rings and Algebras
   \date{\today}

   \maketitle

   \section{Introduction}
   We assume $R$ to be an associative ring with identity $1 \neq 0$. We will represent the Jacobson radical, the nilpotents, the group of units, and the center by $J(R),~N(R),~U(R)$ and $C(R)$, respectively. 

   The set $\sqrt{J(R)}$, defined by Wang and Chen in \cite{wang2012pseudo}, is the set of all those elements $x\in R$ for which there exists some positive integer $m$ satisfying $x^n \in J(R)$. The set $\sqrt{J(R)}$ may fail to be a subring and encompass the nilpotents and the Jacobson radical in it, as seen in \cite[Example 2.1]{UsqrtJrings}. To know more about the set $\sqrt{J(R)}$, \cite{UsqrtJrings, sainiczechoslovak, sqrtJ clean} can be referred to.

   The concept of $UU$ rings, i.e., rings with unipotent units, was introduced by C\u{a}lug\u{a}reanu in \cite{UU}. The $UU$ rings satisfies the equation $U(R)=1+N(R)$. On the contrary, Danchev studied the rings with Jacobson units in \cite{JU} and referred to them as $JU$ rings. A ring $R$ which satisfies $U(R)=1+J(R)$ is called a $JU$ ring. This class was explored as $UJ$ rings by Kosan and others in \cite{UJ}. Recently, Saini and Udar, in \cite{sainiczechoslovak} explored a new class of rings, generalizing both $UU$ and $JU$ rings, and referred to it as: $\sqrt{J}U$ rings. Note that $1+\sqrt{J(R)} \subseteq U(R)$ is always valid. The rings that satisfy the converse, too, are called $\sqrt{J}U$ rings. 

   Generalizing $UJ$ rings, Cui and Yin introduced $2-UJ$ rings. In $2-UJ$ rings, for every unit $u$, $u^2$ is expressible as the addition of $1$ and an element $j\in J(R)$. The rings wherein every unit's square equals the addition of 1 and an element from $N(R)$ are $2-UU$ rings, explored by Shebani and Chen in \cite{2-UU}. Prompted by the above advancements, we introduce $2-\sqrt{J}U$ rings to refine the relationship existing among other ring structures. Hence, we aim to present adequate properties and examples of $2-\sqrt{JU}$ rings by comparing them with other well-known classes of rings.

   \begin{definition}
       A ring $R$ in which the square of every unit equals the sum of 1 and an element from $\sqrt{J(R)}$ is referred to as $2-\sqrt{J}U$ ring, i.e., in a $2- \sqrt{J}U$ ring, for every unit $u$, $u^2 \in 1 + \sqrt{J(R)}$.
   \end{definition}

   Recollect that a semi-boolean ring is the ring in which for any element $a$ in $R$, $a=p+q$, where $p \in Id(R)$ and $q\in J(R)$. By replacing $j\in J(R)$ to $n \in N(R)$ in the above decomposition, we obtain a nil-clean ring. When $ej=je$, the ring $R$ becomes a strongly nil-clean ring. The examples of $2-\sqrt{J}U$ rings are $\mathbb{F}_2$, strongly nil-clean rings, and semi-boolean rings.

   Now, we provide a summary of the findings discussed in the paper: We begin by introducing fundamental properties of $2-\sqrt{J}U$ rings in Section \ref{basic properties}. Next, we also prove that for a $2-\sqrt{J}U$ ring, the corner ring and a unit-closed subring are always a $2-\sqrt{J}U$ ring. We characterize $2-\sqrt{J}U$ division rings and local rings. In the subsequent results, we present a characterization of local $2-\sqrt{J}U$ rings and semisimple $2-\sqrt{J}U$ rings. In Section \ref{matrix extensions}, we demonstrate that any $n \times n$ matrix ring $M_n(R)$, with $n>1$ can never be a $2-\sqrt{J}U$ ring. Morita context and some other matrix rings $T(R,M),~ T_n(R)$ are also examined. We conclude with a study of the Group rings.

   We represent the ring of all $n \times n$ matrices, the ring of all upper triangular matrices, and the set of idempotents of $R$ by $M_n(R)$, $T_n(R)$, and $Id(R)$, respectively. To refer to any other unmentioned terminology, \cite{lam1991first, Sehgal} can be referred to.
   
   \section{Basic properties}\label{basic properties}
Before exploring about $2-\sqrt{J}U$ rings and its relationship with other rings like $2-UU$ and $2-UJ$, we present some properties of $\sqrt{J(R)}$:

   \begin{proposition}\label{basic ppties of sqrt J(R)}
       The following items hold true in a ring $R$:
       \begin{enumerate}
           \item \cite[Proposition 2.5]{sainiczechoslovak} For a $\sqrt{J}U$ ring $R$, $2 \in J(R)$.
           \item \cite[Proposition (v)]{UsqrtJrings} If $x^k \in \sqrt{J(R)}$, then $x\in \sqrt{J(R)}$.
           \item \cite[Proposition (ii)]{UsqrtJrings} $U(R) \cap \sqrt{J(R)} = \phi$ and $\sqrt{J(R)} \cap Id(R)=0$.
       \end{enumerate}
   \end{proposition}

It is worth observing that if $U(R) \subseteq 1+\sqrt{J(R)}$, then for every $u\in U(R)$, $u^2\in 1+ \sqrt{J(R)}$ and hence, every $\sqrt{J}U$ ring is a $2-\sqrt{J}U$ ring. As seen in \cite{sainiczechoslovak}, every $UU$ and $UJ$ ring is $\sqrt{J}U$; by following the similar steps, we can show that every $2-UU$ and $2-UJ$ ring is a $2-\sqrt{J}U$ ring. 

\begin{equation*}
       \begin{matrix}
           {2-UU}& & {2-JU}& & {\sqrt{J}U} \\
           & \searrow& \downarrow & \swarrow & \\
           & & {2-\sqrt{J}U}
       \end{matrix}
   \end{equation*}

These inclusions are proper, and the converse is not always true. We provide the following examples to validate the same.

   \begin{example}
       \begin{enumerate}
           \item The field $\mathbb{F}_3$ is a $2- \sqrt{J}U$ ring which is not a $\sqrt{J}U$ ring.
           \item By considering $R= \mathbb{F}_5<x,y: x^2 =0>$, we obtain that $R$ is a $2-\sqrt{J}U$ ring which fails to be a $2-UJ$ ring. 
           %see 2-UNJ rings Danchev or 2-UNJ rings by P.H. Tin for more details.
           \item Let $T=\mathbb{F}_2[[x]]$, then $T$ is a $2-\sqrt{J}U$ ring by Lemma \ref{R[[t]] is 2 sqrt JU}. If we consider the element $1+x \in T$, then $(1+x)^2=1+2x+x^2=1+x^2$ and also $1+x$ is invertible in $T$. As $(1+x)^2=1+x^2 \notin 1+N(T)$, $T$ fails to be a $2-UU$ ring.
       \end{enumerate} 
   \end{example}

   \begin{proposition} \label{homomorphism}
       For a $2-\sqrt{J}U$ ring R, a homomorphic image is also a $2-\sqrt{J}U$ ring.
   \end{proposition}
   \begin{proof}
       Under the homomorphism $\phi$, for any arbitrary unit $u$, $\phi(u^2)=1+\phi(\sqrt{J(R)})$. The units, as well as elements of $\sqrt{J(R)}$, are invariant under the homomorphism; the proof is clear.
   \end{proof}

   \begin{proposition}\label{product in 2 sqrt J U rings}
       For the rings $R_1$ and $R_2$, $R=R_1 \times R_2$ is a $2-\sqrt{J}U$ ring iff every $R_1, R_2$ are $2- \sqrt{J}U$ rings.
   \end{proposition}
   \begin{proof}
       The proof can be sketched using the results that $\sqrt{J(R_1)} \times \sqrt{J(R_2)}=\sqrt{J(R_1 \times R_2)}$ and $U(R_1 \times R_2)=U(R_1) \times U(R_2)$.
   \end{proof}

   \begin{theorem}\label{R/I is a 2 sqrt J U ring}
       For an ideal $I \subseteq J(R)$, $R/I$ is a $2- \sqrt{J}U$ ring iff $R$ is a $2- \sqrt{J}U$ ring.
   \end{theorem} 
   \begin{proof}
       First, if we assume $R$ is a $2- \sqrt{J}U$ ring, let $\bar{u}$ be invertible in $R/I$. Then $u$ is invertible in $R$ and therefore, $u^2=1+z$, where $z\in \sqrt{J(R)}$. This ensues $\bar{u^2}=\bar{1}+\bar{z}$ [by Propisition \ref{homomorphism}]and hence, we have, $R/I$ is a $2- \sqrt{J}U$ ring.
       
       Next, if $R/I$ is a $2- \sqrt{J}U$ ring, let $u$ be a unit in $R$. Hence, $\bar{u}$ is an invertible element in $R/I$. Hence, $\bar{u^2}=1+\bar{z}$, where $\bar{z}\in \sqrt{J(R/I)}$. This results in $u^2-(1+z)=j\in I \subseteq{J(R)}$. Therefore, we have $u^2=1+(j-z)$, where $j-z \in \sqrt{J(R)}$, proving $R$ is a $2- \sqrt{J}U$ ring, as required.
   \end{proof}

   As a result of the above Theorem, we have the following result:

   \begin{corollary}\label{R/J(R) is a sqrt J U ring}
       A ring $R$ is a $2- \sqrt{J}U$ ring iff $R/J(R)$ is a $2-\sqrt{J}U$ ring.
   \end{corollary}

   If $e$ represents a non-zero idempotent in a ring $R$, then $U(eRe)=eRe \cap (1-e+U(R))$. From \cite[Proposition 2.11]{sainiczechoslovak}, $\sqrt{J(eRe)}=\sqrt{J(R)}\cap eRe=e \sqrt{J(R)} e$. With these results, we proceed to prove the following: 

   \begin{theorem}\label{2 sqrt J U ring passes to corners}
       A $2- \sqrt{J}U$ ring passes to the corner.
   \end{theorem}
   \begin{proof}
       We assume $R$ as a $2- \sqrt{J}U$ ring and $v$ be a unit in $eRe$. Hence, $v\in eRe \cap (1-e+U(R))$. This results in $v+(1-e)$ is a unit in $R$ with $(v+(1-e))^{-1}=v^{-1}+(1-e)$ . Therefore, $(v+(1-e))^2=v^2+(1-e)$ and hence, there exists some $z\in \sqrt{J(R)}$ satisfying $v^2+(1-e)=1+z$. Thus, we have $v^2-e=z\in \sqrt{J(R)}$. By using $\sqrt{J(eRe)}=\sqrt{J(R)}\cap eRe=e \sqrt{J(R)} e$, we obtain, $v^2-e\in eRe \cap \sqrt{J(R)}=\sqrt{J(eRe)}$. Finally, $v^2\in e+\sqrt{J(eRe)}$, as required.
   \end{proof}

   Recall that any ring that satisfies the condition $ab=1$ implies $ba=1$ for any elements $a, ~b \in R$ is called a Dedekind finite ring.

   \begin{corollary}
       A $2- \sqrt{J}U$ ring $R$ is a Dedekind finite ring.
   \end{corollary}
   \begin{proof}
       We begin by assuming a $2- \sqrt{J}U$ ring $R$ in which elements $p, ~q$ exists satisfying $pq=1$ although $qp\neq1$. Next, suppose $e_{ij}=p^i(1-qp)q^j$ and $e=e_{11}+e_{22}+\dots+e_{nn}$. This results in the existence of a ring $R'$ for which the corner ring $eRe$ equals $M_n(R')$. According to Theorem \ref{2 sqrt J U ring passes to corners}, $M_n(R')$ is a $2- \sqrt{J}U$ ring, and due to the Corollary \ref{Mn(R) cannot be 2 sqrt J U},  it is not possible. Therefore, $R$ is a Dedekind finite ring.  \end{proof}

   \begin{lemma}\label{R[[t]] is 2 sqrt JU}
       In a ring $R$, the following items are equivalent:
       \begin{enumerate}
           \item The ring $R$ satisfies $u^2\in 1+ \sqrt{J(R)}$.
           \item $R[[t]]$ is a $2-\sqrt{J}U$ ring.
       \end{enumerate}
   \end{lemma}
   \begin{proof}
       First, note that $(t) \subseteq J(R[[t]])$ and $R \cong R[[t]]/(t)$. The proof can be traced from Theorem \ref{R/I is a 2 sqrt J U ring}.
   \end{proof}

   Now, recall that a ring $R$ is called a unit-closed subring when it satisfies the condition $U(S)=U(R) \cap S$. 

   \begin{proposition}\label{unit closed subring}
       A unit closed subring $S$ of a $2-\sqrt{J}U$ ring $R$ is $2-\sqrt{J}U$.
   \end{proposition}
   \begin{proof}
       As the subring $S$ is a unit closed subring, by \cite[Lemma 2.7]{UsqrtJrings}, $\sqrt{J(R)} \cap S \subseteq \sqrt{J(S)}$. Let $u$ be a unit in $S$. Hence, $u\in U(R)$ and therefore, given that $R$ is a $2- \sqrt{J}U$ ring, we have, $u^2 \in 1+\sqrt{J(R)}.$ This ensues $u^2 \in (1+ \sqrt{J(R)})\cap S \subseteq 1+ \sqrt{J(S)}$. In particular, this proves that $S$ is a $2- \sqrt{J}U$ ring.
   \end{proof}

   \begin{theorem}\label{division ring which is 2 sqrt J U}
       Let $R$ be a division ring. Then $R$ is a $2- \sqrt{J}U$ ring iff $R \cong \mathbb{F}_2$ or $\mathbb{F}_3$.
   \end{theorem}
   \begin{proof}
       In a division ring, $\sqrt{J(R)}$ equals 0. If we consider $R$ as a $2- \sqrt{J}U$ ring, then every unit is self-invertible, satisfying $u^2=1$. The remaining follows from \cite[Example 2.1]{2-UJ}. The converse part is straightforward.
   \end{proof}

   A ring $R$ is called a local ring if $R/J(R)$ is a division ring.

   \begin{corollary}
       A local ring $R$ is a $2- \sqrt{J}U$ ring iff $R/J(R) \cong \mathbb{F}_2$ or $\mathbb{F}_3$.
   \end{corollary}
   \begin{proof}
   From Corollary \ref{R/J(R) is a sqrt J U ring} and Theorem \ref{division ring which is 2 sqrt J U}, the proof is clearly followed.
   \end{proof}

   Recall that by Wedderburn Artin's theorem, a semisimple ring $R$ is isomorphic to $\prod_{i=1}^{k} M_{m_i}(D_i)$ for appropriate division rings $D_i$ and a unique positive integer $k$.

   \begin{theorem}
       A semisimple ring $R$ is a $2-\sqrt{J}U$ ring iff $\prod R_i$, where each $R_i$ is either isomorphic to $\mathbb{F}_2$ or $\mathbb{F}_3$.
   \end{theorem} 
   \begin{proof}
       From Wedderburn Artin's theorem, a semisimple ring $R$ is isomorphic to $\prod_{i=1}^{k} M_{m_i}(D_i)$ for appropriate division rings $D_i$ and a unique positive integer $k$. Hence, if $R$ is a $2- \sqrt{J}U$ ring, by Proposition \ref{product in 2 sqrt J U rings}, the squares of the units are in the set $1+\sqrt{J(M_{m_i}(D_i))}$. Then, by Corollary \ref{Mn(R) cannot be 2 sqrt J U}, we obtain $m_i=1$ for every $i$. Hence, by Theorem \ref{division ring which is 2 sqrt J U}, each $M_i \cong \mathbb{F}_2$ or $\mathbb{F}_3$. The converse part is obvious.
   \end{proof}

   \begin{proposition}
       $R$ is a $\sqrt{J}U$ ring iff $2\in J(R)$ and $R$ is a $2- \sqrt{J}U$ ring.\\
       $R$ is a $\sqrt{J}U$ ring iff $R$ is a $2-\sqrt{J}U$ ring with $2 \in J(R)$
   \end{proposition}
   \begin{proof}
      (1)$\Rightarrow$(2)

       The proof is coherent because in a $\sqrt{J}U$ ring $2\in J(R)$ and every $\sqrt{J}U$ ring is a $2-\sqrt{J}U$ ring.\\
       (2)$\Rightarrow$(1)

      In a $2-\sqrt{J}U$ ring, for any unit $u$, there exists some $z\in \sqrt{J(R)}$ satisfying $u^2\in1+\sqrt{J(R)}$. Thus, we have $u^2-1 \in \sqrt{J(R)}$. Given that $2\in J(R)$, $2(u-1) \in J(R)$. By \cite[Lemma 2.11(1)]{u-1 lies in sqrt J(R)}, we obtain that $(u^2-1)-2(u-1) \in \sqrt{J(R)}$. Consequently, we get $u^2-2u+1=(u-1)^2 \in \sqrt{J(R)}$. Following Proposition \ref{basic ppties of sqrt J(R)}(2), $u-1 \in \sqrt{J(R)}$, as required.
  \end{proof} 

   \begin{lemma}
           A ring $R$ is a $2- \sqrt{J}U$ ring iff $R[x]/x^pR[x]$ is a $2- \sqrt{J}U$ ring, with the positive integer $p$.
   \end{lemma}
   
   \begin{proof}
       The proof follows from $(R[x]/x^p R[x])/(xR[x]/x^pR[x]) \cong R$ and by using Theorem \ref{R/I is a 2 sqrt J U ring} as $xR[x]/x^pR[x] \subseteq J(R)$.
   \end{proof}

   For a ring $R$, if $\alpha$ denotes a ring endomorphism on $R$, then $R[[x;\alpha]]$ is the ring of skew formal power series over $R$. Here, the addition is component-wise, and multiplication is determined by:
   \begin{equation*}
       xr= \alpha(r)x ~ \forall ~ r\in R.
   \end{equation*}
   When $\alpha = 1$, then $R[[x; \alpha]]=R[[x]]$.

   \begin{proposition}
       Any ring $R$ is a $2- \sqrt{J}U$ ring if and only if $R[[x; \alpha]]$ is a $2- \sqrt{J}U$ ring.
   \end{proposition}
   \begin{proof}
       Let $(x)=R[[x: \alpha]]x$. Then $(x) \subseteq J(R[[x; \alpha]])=J(R)+I$. As $R[[x; \alpha]]/(x) \cong R$, the proof follows from Theorem \ref{R/I is a 2 sqrt J U ring}.
   \end{proof}

Note that Lemma \ref{R[[t]] is 2 sqrt JU} can be proved again using the above Proposition.
   %\begin{corollary}
    %   A ring $R$ is a $2- \sqrt{J}U$ iff $R[[x]]$ is a $2- \sqrt{J}U$ ring.
   %\end{corollary}

   \begin{proposition}
       For units $u$ and $v$ in a $2- \sqrt{J}U$ ring $R$, $u^2+v \neq 1$.
   \end{proposition}
   \begin{proof}
       If possible, we assume there exist units $u, v$ satisfying $u^2+v=1$. Now, due to $R$ being a $2-\sqrt{J}U$ ring, we have $u^2 \in 1 + \sqrt{J(R)}$. Let $u^2=1+z$, where $z \in \sqrt{J(R)}$. Hence, $1=u^2+v=(1+z)+v$. For this reason, we obtain $v=-z \in \sqrt{J(R)}$, which is a contradiction from Proposition \ref{basic ppties of sqrt J(R)} (3).
   \end{proof}

   \begin{proposition}\cite[Lemma 5.1 (2)]{2-UNJ rings}\label{intersection of sqrt J(R) and C(R)}
       For a ring $R$, $\sqrt{J(R)} \cap C(R) \subseteq J(R)$.
   \end{proposition}
   
   \section{matrix extensions}\label{matrix extensions}

\begin{theorem}\label{Mn(R) cannot be 2 sqrt J U}
       For a ring $R$, the matrix ring $M_n(R)$ is a $2- \sqrt{J}U$ ring iff $n\geq 2$.
   \end{theorem}
   \begin{proof}
       We will first prove that $M_2(R)$ can never be a $2- \sqrt{J}U$ ring. Consider a unit $\begin{pmatrix}
           1 & 1 \\ 1 & 0
       \end{pmatrix}$ form $M_2(R)$. Then \begin{equation*}
           A^2-I=\begin{pmatrix}
               1 & 1 \\ 1 & 0
           \end{pmatrix}.
       \end{equation*}
       Hence, we obtain that $A^2-I$ is an invertible element in $M_2(R)$. Therefore, from \cite[Proposition 2.2 (ii)]{UsqrtJrings}, $A \notin \sqrt{J(M_2(R))}$. This indicates that $M_2(R)$ cannot be a $2- \sqrt{J}U$ ring. As the corner ring of $n \times n$ matrix ring $M_n(R)$ is isomorphic to $M_2(R)$, and $M_2(R)$ cannot be a $2~ \sqrt{J}U$ ring, $M_n(R)$ is not a $2- \sqrt{J}U$ ring for any positive integer $n$.
   \end{proof}

If $R$ is a ring, then $T(R,R,M)=\begin{pmatrix}
    R & M\\ 0 & R
\end{pmatrix}$ denotes the formal matrix ring. The set $T(R,M)=\begin{pmatrix}
    x & n \\ 0 &x
\end{pmatrix} : x\in R, n\in M$ is a subring of $T(R,R,M)$. Notably,$T(R,M)$ is ismorphic to the ring $R \propto M$, where $R \propto M=\{(a,b): a \in R; b \in M \}$. The addition is formulated component-wise. For any $ (x,m), (y,n) \in R \propto M$, multiplication is defined as:
\begin{equation*}
    (x,m)(y,n)=(xy,xn+my).
\end{equation*}
The units in $T(R,M)$ is:
   \begin{equation*}
    U(T(R,M))=\{(u,n): u\in U(R), n\in M \}.
\end{equation*}
The Jacobson radical is represented by: 
\begin{equation*}
    J(T(R,M))=\{ (j,n) : j\in J(R), n\in M \}.
\end{equation*}
It can be easily observed that 
\begin{equation*}
    \sqrt{J(T(R,M))}=\{ (z,n) : z\in J(R), n\in M \}.
\end{equation*}

\begin{theorem}\label{T(R,M) is 2 sqrt(J)U}
    Let R be a ring. Then the following are equivalent:
    \begin{enumerate}
        \item R is a $2-\sqrt{J}U$ ring.
        \item $T(R,M)$ is a $2-\sqrt{J}U$ ring.
        \item $R \propto M$ is a $2-\sqrt{J}U$ ring.
    \end{enumerate}
\end{theorem}
\begin{proof}
    (1)$\Rightarrow$(2)

    Let $u$ be invertible in $R$. Then a $z\in \sqrt{J(R)}$ exists satisfying $u^2=1+z \in 1 + \sqrt{J(R)}$. This helps in creating an invertible matrix $U=\begin{pmatrix}
        u & m \\ 0 & u
    \end{pmatrix}$. Then 
    \begin{align*}
        \begin{pmatrix}
            u & m \\ 0 & u
        \end{pmatrix}^2 - \begin{pmatrix}
            1 & 0 \\ 0 & 1
        \end{pmatrix}&=\begin{pmatrix}
            u^2-1 & um+mu \\ 0 & u^2-1
        \end{pmatrix} \\ & = \begin{pmatrix}
            z& um+mu \\ 0 & z
        \end{pmatrix}
    \end{align*}
    As $z\in \sqrt{J(R)}$ and $um+mu \in M$, we get $U^2-I \in \sqrt{J(T(R,M))}$, as required.

    (2)$\Rightarrow$(1)

    If we assume $T(R,M)$ is a $2-\sqrt{J}U$ ring, then let $U=\begin{pmatrix}
        u & m \\ 0 & u
    \end{pmatrix}$ be a unit in $T(R,M)$, hence we have \begin{equation*}
        \begin{pmatrix}
            u & m \\ 0 & u
        \end{pmatrix}^2 - \begin{pmatrix}
            1 & 0 \\ 0 & 1
        \end{pmatrix} \in \sqrt{J(T(R,M))}.
    \end{equation*} This gives the matrix $\begin{pmatrix}
        
            u^2-1 & um+mu \\ 0 & u^2-1
        \end{pmatrix}
    $ is an element in $\sqrt{J(T(R,M))}$. Thus, $u^2-1 \in \sqrt{J(R)}$, as required.  
\end{proof}

The remaining part (2)$\iff$(3) is clear as $T(R,M) \cong R \propto M$.

\begin{proposition}
    A ring $R$ is called a $2-\sqrt{J}U$ ring is equivalent to:
       \begin{enumerate}
           \item $R \propto R$ is $2-\sqrt{J}U$.
           \item $T(R,R)$ is $2-\sqrt{J}U$.
           \item $R[x]/(x^2)$ is $2-\sqrt{J}U$.
       \end{enumerate} 
\end{proposition}

Assume $X$ and $Y$ are the rings and $_XM_Y$ and $_YN_X$ are the bimodules. Now, recalling the Morita context, we describe the context product. They are the mappings $ M \times  N \to X$ and $N \times M \to Y$ defined by $(m,n) \mapsto mn$ and $(n,m) \mapsto nm$, respectively. Hence, by usual matrix operations, the Morita context is $\begin{pmatrix}
    X & M \\ N & Y
\end{pmatrix}$. If the context products are trivial, that is, $MN$ and $NM$ are both $0$, then the Morita context is known as a trivial Morita context. The $\begin{pmatrix}
    X & M \\ N & Y
\end{pmatrix}$ is isomorphic to $T(X \times Y, M \oplus N)$. Examples are formal triangular matrices and $T_n(R)$.

   \begin{proposition}
    Suppose M is a bimodule over the rings $R_1, ~ R_2$. Then $R_1$ and $R_2$ are $2-\sqrt{JU}$ rings, provided $\begin{pmatrix}
        R_1 & M \\ 0 & R_2
    \end{pmatrix}$ is a $2-\sqrt{J}U$ ring.
\end{proposition}
\begin{proof}
    The proof follows from Theorem \ref{T(R,M) is 2 sqrt(J)U} and Proposition \ref{homomorphism} as $\begin{pmatrix}
        R_1 & M \\ 0 & R_2
    \end{pmatrix} \cong T(R_1 \times R_2, M)$.
\end{proof}

\begin{proposition}
    In a ring R and $n\geq 1$, $T_n(R)$ is a $2-\sqrt{J}U$ ring implies $R$ is $2-\sqrt{J}U$ ring.
\end{proposition}

Now, we define $BT(R, M)$. For, assume a ring $R$ and a bimodule $M$ over $R$. Then, $BT(R,M)=\{ (x,p,y,q): x,y \in R; p,q \in M\}$. With the operations of component-wise addition and multiplication described as:
\small{\begin{equation*}
    (x_1,p_1,y_1,q_1)(x_2,p_2,y_2,q_2)=(x_1x_2, x_1p_2+p_1x_2, x_1y_2+y_1x_2, x_1q_2+p_1y_2,y_1p_2,q_1x_2),
\end{equation*}}
$BT(R,M)$ forms a ring.
Additionally, $BT(R,M)\cong T(T(R,M), T(R,M))$. We also have:
\begin{equation*}
    BT(R,M)= \left \{ \begin{pmatrix}
        x & p & y & q \\ 0 & x & 0 & y \\ 0 & 0 & x & p\\ 0 & 0 & 0 & x    \end{pmatrix} : x,y \in R ; p,q \in M \right \}.
\end{equation*}
By using the map $m+nx+py+qxy \to \begin{pmatrix}
     m&n&p&q\\
     0&m&0&p\\
     0&0&m&n\\
     0&0&0&m
 \end{pmatrix}$, we obtain $R[x,y]/(x^2,y^2) \cong BT(R,R)$.

 \begin{theorem}
     Let R be a ring and M be a bimodule over R. Then the following are equivalent:
     \begin{itemize}
         \item[(i)] R is a $2-\sqrt{J}U$ ring.
         \item[(ii)] $R \bowtie M$ is a $2-\sqrt{J}U$ ring.
         \item[(iii)] BT(R,M) is a $2-\sqrt{J}U$ ring.
         \item[(iv)] BT(R,R) is a $2-\sqrt{J}U$ ring.
         \item[(v)] $R[x,y]/(x^2,y^2)$ is a $2-\sqrt{J}U$ ring.
         \item[(vi)] $R \bowtie R$ is a $2-\sqrt{J}U$ ring.
     \end{itemize}
 \end{theorem}

 \begin{proof}
     The results are clear following Theorem \ref{T(R,M) is 2 sqrt(J)U}.
 \end{proof}

   \section{group rings}

   \begin{proposition}
       When $RG$ is a $2-\sqrt{J}U$ ring, so is the ring $R$.
   \end{proposition}
   \begin{proof}
       We start with assuming that $RG$ is a $2-\sqrt{J}U$ ring and $v$ is an invertible element in $R$. Hence, $v$ is a unit in $RG$. As a result, we obtain that $v^2 \in 1+\sqrt{J(RG)}$. This simplifies to $v^2 \in (1+\sqrt{J(RG)}) \cap R$. Consequently, $v^2 \in 1 + (\sqrt{J(RG)} \cap R) \subseteq 1+ \sqrt{J(R)}$, which proves $R$ is a $2-\sqrt{J}U$ ring, as required.
   \end{proof}

   \begin{proposition}
       For a subgroup $H$ of a group $G$ and a ring $R$, $RH$ is a $2-\sqrt{J}U$ ring if $RG$ is a $2-\sqrt{J}U$ ring.
   \end{proposition}
   \begin{proof}
       For the subgroup $H$ of group $G$, it is well known that $RH$ is a unit-closed subring of $RG$. Thereby, following Proposition \ref{unit closed subring}, it follows that $RH$ is a $2-\sqrt{J}U$ ring.
   \end{proof}

   \begin{theorem}
       If $RG$ is a $2 - \sqrt{J}U$ ring and $2$ is an element in the Jacobson radical, then the group $G$ is a $2$-group.
   \end{theorem}
   \begin{proof}
       Proceeding by the method of contradiction, we assume that $G$ is not a 2-group. Hence, let $g\in G$ be such that the order of $g$ is not divisible by 2. This provides the order of $g$ in the form of $2n+1$, where $n$ is a positive integer. Hence, $g^{2n+1}-1=0.$ This gives :
       \begin{equation*}(g-1)(1+g+g^2+ \dots + g^{2m})=0. \end{equation*} \\
       We claim that for any positive integer $m$, 
       \begin{equation*}
           1+g+g^2+ \dots + g^{2m} \in U(RG).
       \end{equation*}
      We first establish the result for $m=1$. That is, we aim to show that $1+g+g^2$ is a unit in $RG$. Since, $g \in U(RG)$ and $RG$ is a $2-\sqrt{J}U$ ring, $g^2-1 \in \sqrt{J(RG)}$. Furthermore, we can also note that $g^2 -1 \in C(RG)$. Thus, from Proposition \ref{intersection of sqrt J(R) and C(R)}, 
      \begin{equation*}
          1-g^2 \in J(RG).
      \end{equation*}
      Now, as $2\in J(R)$, based on the definition of $J(R)$, we have, $2^2-1=3$ is a unit in $R$ and accordingly, $3 \in U(RG)$. While $RG$ is a $2-\sqrt{J}U$ ring, $3^2-1=2^3\in \sqrt{J(RG)}$. Hence, from Proposition \ref{basic ppties of sqrt J(R)} (2), $2 \in \sqrt{J(RG)}$ also. Thus, 2 is an element of $\sqrt{J(RG)}$ which is also central in $RG$. Hence, from Proposition \ref{intersection of sqrt J(R) and C(R)}, $2\in J(RG)$. This results in $2g^2 \in J(RG)$. Hence,
      \begin{equation*}
          1+g^2=(1-g^2)+2g^2 \in J(RG).
      \end{equation*}
      From the above developments, $1+g+g^2=g+(1+g^2) \in U(RG)+J(RG)$ and by using $U(RG)+J(RG) \subseteq U(RG)$, we obtain $1+g+g^2$ is a unit in $RG$.

      If $m=2$, we need to establish that $1+g+g^2+g^3+g^4$ is a unit in $RG$. As \begin{equation*}
          1+g+g^2+g^3+g^4=(1+g^2)(1+g)+g^4,
      \end{equation*}
      and $(1+g^2)(1+g) \in J(RG)$ and $g^4$ is a unit, $\sum_{i=0}^{i=4} g_i$ is a unit in $RG$. 

      In a similar manner, our claim can be proved true for every positive integer. As a result, the equation 
      \begin{equation*}
          g^{2n+1}-1=(g-1)(1+g+g + \dots + g^{2m})=0
      \end{equation*}
      giving $g$ as the identity element $1$. Hence, we have arrived at a contradiction, and $G$ is a 2-group. 
      
    \end{proof}
When every element of the group possesses a finite order, we refer to the group as a torsion group.
%When the order of every $a\in G$ is finite, the group $G$ is referred to as a torsion group.
   \begin{theorem}
       For a $2-\sqrt{J}U$ group ring $RG$, the group $G$ is always a torsion group.
   \end{theorem}
   \begin{proof}
       Proceeding by method of contradiction, assume that $G$ is not a torsion group and hence, an element $g$ exists in $G$ whose order is infinite. As $g\in RG$ is invertible in a $2- \sqrt{J}U$ ring $RG$, $g^2-1 \in \sqrt{J(RG)}$. Hence, for some positive integer $n$, $g^2-1 \in J(RG)$. Also, $g^-1 \in R<g>.$ Hence, $(g^2-1)^n \in J(RG) \cap R<g> \subseteq J(R<g>)$, which results in $g^2-1 \in \sqrt{J(R<g>)}$. It is clear that $g^2-1$ is a central element in $R<g>$. As $(g^2-1)^n \in J(R<g>)$, we have $1-r^n((g^2-1)^n)\in U(R<g>)$, for any $r$ in $R<g>$. Observe that $r (g^2-1)=(g^2-1)r$ and hence
       \begin{equation*}
           1-r^n(g^2-1)^n= (1-r(g^2-1))(1+r(g^2-1)+(r(g^2-1))^2+ \dots + (r(g^2-1))^{n-1}).
       \end{equation*}
       This gives $1-r(g^2-1) \in U(R<g>)$ and hence, $g^2-1 \in J(R<g>)$. Hence, we obtain that \begin{equation*}
           (1-g^2)+(2g^2-g) \in J(R<g>)+ U(R<g>)
       \end{equation*}
       and thus, we get $1-g+g^2$ as an invertible element. Thus, there exists non zero elements $a_n$ and $a_m$ in $R$ for some positive integers $n $ and $m$, where $n<m$ satisfying $(1-g+g^2) ( \sum_{i=1}^{m}a_ig^i)=1$, which is a contradiction. Hence, $G$ is a torsion group, as required. 
   \end{proof} 
   
When any subgroup generated by a finite number of elements is itself a finite group, the group is referred to as a locally finte group.

   \begin{theorem}
       In the group ring $RG$, for every unit $u$, $u^2 \in 1+\sqrt{J(R)}$ if it satisfies the following:
       \begin{enumerate}
           \item $R$ is a $2-\sqrt{J}U$ ring.
           \item $3 \in J(R)$.
           \item the group $G$ is a locally finite group $2$-group.
       \end{enumerate}
   \end{theorem}
   \begin{proof}
       Let $f$ be a unit in $RG$. There will exist a finite subgroup $H$ of $G$ in which $f$ is a unit. As $G$ is a $2$-group, $H=\oplus_{i=1}^{n}C_2$ and hence, $R = \oplus_{i=1}^{2n} R$, which is a $2-\sqrt{J}U$ ring from Proposition \ref{product in 2 sqrt J U rings}. Hence, $f^2=1+z$, $z\in \sqrt{J(RH)}$. Hence, for some positive integer $m$, $(f^2-1)^m \in J(RH)$. Next, by \cite[Proposition 9]{Connel}, as $J(RG) \cap RH \subseteq J(RH)$, we have $(f^2-1)^m \in J(RH) \subseteq J(RG)$. This resulta in $f^2-1 \in \sqrt{J(RG)}$, as required.
   \end{proof}

\end{document}